\documentclass{article}
\usepackage{amsthm}

\newtheorem{theorem}{Theorem}
\newtheorem{lemma}[theorem]{Lemma}
\newtheorem{corollary}[theorem]{Corollary}
 
\theoremstyle{definition}
\newtheorem*{definition}{Definition}
\newtheorem*{remark}{Remark}

\usepackage{tikz}
\usetikzlibrary{math}
\usepackage{subcaption}
\usepackage{mathtools}

\begin{document}

\title{Block Stacking, Airplane Refueling, and\\Robust Appointment Scheduling}
\markright{}
\author{Simon Gmeiner and Andreas S.\ Schulz}
\date{}

\maketitle

\begin{sloppypar}
\begin{abstract}
How can a stack of identical blocks be arranged to extend beyond the edge of a table as far as possible? We consider a generalization of this classic puzzle to blocks that differ in width and mass. Despite the seemingly simple premise, we demonstrate that it is unlikely that one can efficiently determine a stack configuration of maximum overhang. Formally, we prove that the Block-Stacking Problem is NP-hard, partially answering an open question from the literature. Furthermore, we demonstrate that the restriction to stacks without counterweights has a surprising connection to the Airplane Refueling Problem, another famous puzzle, and to Robust Appointment Scheduling, a problem of practical relevance. In addition to revealing a remarkable relation to the real-world challenge of devising schedules under uncertainty, their equivalence unveils a polynomial-time approximation scheme, that is, a $(1+\epsilon)$-approximation algorithm, for Block Stacking without counterbalancing and a $(2+\epsilon)$-approximation algorithm for the general case.
\end{abstract}
\end{sloppypar}

\section{Introduction.} \label{section:introduction}
What is the largest possible overhang that can be achieved when stacking a set of blocks? This is the core question of the Block-Stacking Problem, a puzzle that has been around in various forms for decades (see, e.g., \cite{paterson2009maximum} for references dating back to the middle of the 19th century). 
Typically, the problem is presented assuming all blocks are identical and can only be stacked one on top of the other.
Remarkably, it is theoretically possible to construct balanced stacks that achieve an arbitrarily large overhang, even in this highly constrained form. 
As a result, the problem has attracted a great deal of attention. It has become a widespread puzzle and experiment that appears in many textbooks and countless videos on the Internet, albeit inconsistently under different names such as ``Leaning Tower of Lire'' (e.g.,~\cite{johnson1955leaning}), ``Block-'' or ``Book-Stacking Problem,'' ``Building an Infinite Bridge,'' and the like.
Besides blocks and books, other objects such as bricks, coins, or playing cards can and have been used to build such stacks. The underlying mechanics, however, always remains the same.

\begin{figure}[ht] 
\centering
\begin{subfigure}[b]{.48 \textwidth}
\centering
\includegraphics{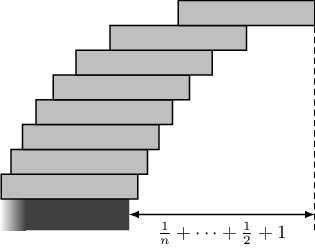}
\end{subfigure}
\begin{subfigure}[b]{.48 \textwidth}
\centering
\includegraphics{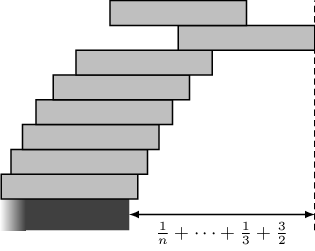}
\end{subfigure}
\caption{The harmonic stack and the alternative balanced stack configuration with maximum overhang for identical blocks of width two.}
\label{figure:identical_blocks}
\end{figure}

It has long been known (see, e.g.,~\cite{hall2005fun,johnson1955leaning,paterson2009maximum}) that for $n$ identical blocks of width two, one can obtain a balanced stack with overhang $\sum_{k=1}^{n} 1/k$ by iteratively placing the center of gravity of the top $i$ blocks right at the edge of the block below, for all~$i=1, \ldots, n$. 
Since the harmonic series, $\sum_{k=1}^{\infty} 1/k$, diverges, the overhang of this so-called harmonic stack, shown in Figure~\ref{figure:identical_blocks} on the left, can become arbitrarily large for sufficiently large~$n$. 
Although it had long been surmised that this construction method yields the maximum possible overhang, it was only recently that Treeby~\cite{treeby2018further} provided rigorous and conclusive proof. 
However, as already pointed out by Hall~\cite{hall2005fun}, this solution is, perhaps surprisingly, not unique because the same overhang can be achieved by letting the second-highest block protrude while using the uppermost block as a counterweight (cf.\ Figure~\ref{figure:identical_blocks}).
Hall~\cite{hall2005fun} further pursued the concept of counterbalancing and discovered that allowing multiple blocks per layer can achieve an even larger overhang with the same number of blocks, a variant Paterson et al.~\cite{paterson2009maximum,paterson2009overhang} later studied in detail.
Other modifications, such as blocks with an uneven weight distribution~\cite{horton1997leaning,kazachkov2017stack}, sticky blocks~\cite{lengvarszky2023maximum}, or a continuous version~\cite{polster2012case}, have also been explored. \\

Treeby~\cite{treeby2018further} was the first to generalize the Block-Stacking Problem beyond a set of identical blocks. He continued to require that all blocks have the same height but allowed them to differ in width. 
More specifically, he assumed that the mass of a block is proportional to its width.
Since the blocks are no longer interchangeable in this setting, the order in which they are stacked affects the achievable overhang. 
When no counterbalancing is allowed, the optimal stacking order is obtained by simply sorting the blocks by width, with the widest block on top (see Figure~\ref{figure:treeby_blocks}).
However, unlike for identical blocks, counterweights can significantly increase the overhang.
This makes finding the maximum possible overhang much more difficult, as significantly more complex stack configurations must be considered.
Treeby even conjectured it to be NP-hard to find an optimal solution. For general information on NP-hardness and computational complexity, see Section~\ref{section:np_hardness}. 

\begin{figure}[ht] 
\centering
\begin{subfigure}[b]{.48 \textwidth}
\centering
\includegraphics{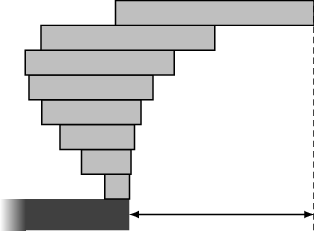}
\end{subfigure}
\begin{subfigure}[b]{.48 \textwidth}
\centering
\includegraphics{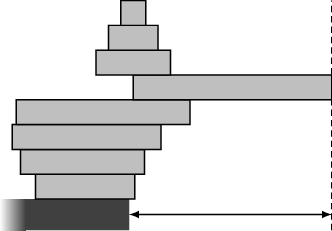}
\end{subfigure}
\caption{Without counterbalancing, blocks of the same height, ordered from top to bottom by non-increasing width, achieve the maximum possible overhang. Using some of the blocks as counterweights allows for an even larger overhang.}
\label{figure:treeby_blocks}
\end{figure}

We go one step further than Treeby by considering differences in height, material, or three-dimensional depth between the blocks.
Since only the center of gravity with respect to the cross-section determines whether a stack is balanced, none of these factors directly affects the overhang of a stack. Instead, their effects can be aggregated into the individual mass of a block, which in turn has a tremendous impact on finding the ideal stack configuration.

As seen in Figure~\ref{figure:treeby_blocks}, we usually want the widest block to protrude. However, if that block is disproportionately heavy, placing it further down the stack might be better, as shown in Figure~\ref{figure:general_blocks}. 
A short and heavy block, by contrast, should ideally either be located toward the bottom of the stack or serve as a counterweight to extend the overhang of the protruding block.
The Block-Stacking Problem aims to identify a stack configuration that yields the largest possible overhang. 

\begin{figure}[ht] 
\centering
\begin{subfigure}[b]{.48 \textwidth}
\centering
\includegraphics{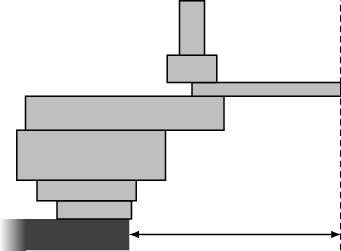}
\end{subfigure}
\begin{subfigure}[b]{.48 \textwidth}
\centering
\includegraphics{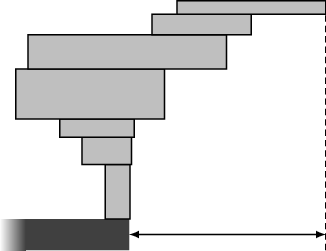}
\end{subfigure}
\caption{Optimal stack configurations with and without counterbalancing. In this and all subsequent figures, we assume that the mass of a block is proportional to the product of its width and height.}
\label{figure:general_blocks}
\end{figure}

In Section~\ref{section:np_hardness}, we will prove that due to the additional complexity introduced by arbitrary widths and masses, the Block-Stacking Problem is, in fact, NP-hard. 
That is, unless P = NP, a solution cannot be found by an efficient (i.e., polynomial-time) algorithm, let alone by a simple rule. To some extent, this is a response to Treeby's conjecture formulated for blocks of the same height. 
Additionally, this highlights that a seemingly simple problem, such as stacking blocks, can be surprisingly challenging.
Nevertheless, in Section~\ref{section:approximation}, we provide a $(2+\epsilon)$-approximation algorithm. An $\alpha$-approximation algorithm denotes an efficient algorithm that returns a solution whose objective value (in our case, the overhang) deviates from the optimal one by at most a factor~$\alpha$.

Unlike for blocks of the same height, even when no counterbalancing is allowed, no simple method is known for determining the maximum possible overhang.
Nonetheless, there is a close connection between the Block-Stacking Problem, the Airplane Refueling Problem, and the Robust Appointment Scheduling Problem, three optimization problems with entirely different backgrounds.

The Airplane Refueling Problem is a mathematical puzzle that asks to maximize the range of a fleet of aircraft that can share their fuel in flight. A detailed description of the problem can be found in Section~\ref{section:ar}. 
Due to its connection to single-machine scheduling problems with a concave cost function (see~\cite{vasquez2015airplane} for more details), it has lately attracted some attention after being mentioned in~\cite{woeginger2010airplane} as a problem of unknown computational complexity. 
Whether the Airplane Refueling Problem can be solved efficiently remains an open question. However, Gamzu et al.~\cite{gamzu2019polynomial} recently developed a polynomial-time approximation scheme~(PTAS), providing a $(1+\epsilon)$-approximation algorithm for any fixed~$\epsilon > 0$.
In addition, dominance and precedence conditions have been investigated~\cite{vasquez2015airplane}, based on which a branch-and-bound algorithm has been devised~\cite{li2019fast}.
Furthermore, a mixed-integer linear programming model with $O(n^2)$ constraints has been proposed~\cite{zhang2021novel}.
As the latter two methods outperform naive approaches in terms of running time, they help solve practical cases. However, they are not considered efficient in the sense of complexity theory (cf.~Section~\ref{section:np_hardness}). Despite its name, the Airplane Refueling Problem has thus far been primarily an academic puzzle with no immediate real-world applications.

In contrast, Appointment Scheduling is an optimization problem with practical background, where the goal is to devise an ideal schedule for tasks whose time requirements can only be estimated (see Section~\ref{section:ras} for details). 
As an alternative to several stochastic models, Mittal et al.~\cite{mittal2014robust} introduced the problem of Robust Appointment Scheduling. 
As pointed out in~\cite{mittal2014robust,schulz2019robust}, similar to the Airplane Refueling Problem, Robust Appointment Scheduling is related to single-machine scheduling problems with a concave cost function. This allowed some immediate results regarding approximation algorithms.
Good approximation guarantees can already be obtained by simple ratio-based algorithms~\cite{hohn2015performance,schulz2019robust}, and, as for the Airplane Refueling Problem, there exists a PTAS~\cite{megow2018dual,stiller2010increasing}. 
However, for Robust Appointment Scheduling, it also remains an open question whether an optimal solution can be found efficiently.

Despite their different origins and even though little is known about whether it is possible to solve them efficiently, in Sections~\ref{section:ar} and~\ref{section:ras}, we prove that Block Stacking without counterbalancing, Airplane Refueling, and Robust Appointment Scheduling are in some sense equivalent and thus form a group of problems with the same computational complexity. 
As a result, the existing algorithms for the Airplane Refueling Problem can also be used to solve Robust Appointment Scheduling and Block Stacking without counterbalancing.
Moreover, the PTAS for the Airplane Refueling Problem can be directly transferred to the Block-Stacking Problem. Most importantly, an efficient algorithm for any of these three problems would imply such an algorithm for the other two problems.

\section{The Block-Stacking Problem.} \label{section:bsp}

We consider a set $B = \{1,...,n\}$  of blocks, each characterized by its half-width~$w_i \geq 0$ and mass~$m_i > 0$. The goal is to arrange these blocks into a balanced stack that extends as far as possible beyond the edge of a table. For now, let the stacking order be fixed and the blocks be lined up sequentially from top to bottom.
For each block~$i$, let $x_i$ denote the horizontal position of its midpoint, relative to the edge of the table. 
Moreover, define
\[ M_i = \sum_{j :\, j \leq i} m_j \] 
to be the mass of the top~$i$ blocks. 
In accordance with the laws of physics, a stack is \textit{balanced} if and only if for every $i = 1,\dots,n-1$, the center of gravity of the blocks~$1$~through~$i$, $\bar{x}_i =  ( \sum_{j :\, j\leq i} m_j x_j ) / M_i$, is located within the edges of \mbox{block $i+1$}, i.e.,
$x_{i+1} - w_{i+1} \leq \bar{x}_i \leq x_{i+1} + w_{i+1}$.
Furthermore, without loss of generality, we assume that the overall center of gravity is positioned precisely at the edge of the table. By convention, the overhang of a stack is always built to the right.

As seen for identical blocks (cf.\ Figure~\ref{figure:identical_blocks}), a natural strategy is to iteratively place the top $i$ blocks as far to the right as possible without causing the stack to topple. 
As suggested by Treeby~\cite{treeby2018further}, we say that a stack is \textit{right-aligned} at block $i$ if the center of gravity of blocks $1$ through $i$ is lined up with the right edge of block \mbox{$i+1$}, i.e., if \mbox{$\bar{x}_i = x_{i+1} + w_{i+1}$} (see Figure~\ref{figure:notation}).
However, as mentioned earlier, the optimal stack configuration for non-identical blocks may not be right-aligned at every block. 
Hence, the maximum overhang is not necessarily attained by the uppermost block but is instead given by $\max_{i} (x_i + w_i)$. 
We refer to the block that reaches this overhang with its right edge as the \textit{protruding} block.

\begin{figure}[ht] 
\centering
\begin{subfigure}[b]{.48 \textwidth}
\centering
\includegraphics{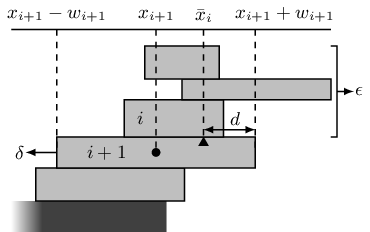}
\end{subfigure}\
\begin{subfigure}[b]{.48 \textwidth}
\centering
\includegraphics{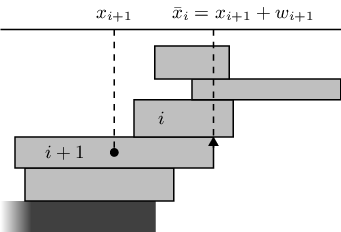}
\end{subfigure}
\caption{Both stacks are balanced, but only the one on the right is right-aligned at block $i$. The indicated adjustments for the left stack as per Theorem~\ref{result:theorem:stack_configuration} lead to a larger overhang. }
\label{figure:notation}
\end{figure}

We are interested in finding a stack with the largest possible overhang. By generalizing Proposition 5.1 of Treeby~\cite{treeby2018further}, one can explicitly describe the structure that any such stack must have. This structure can already be observed in Figures~\ref{figure:identical_blocks},~\ref{figure:treeby_blocks}, and~\ref{figure:general_blocks}. 

\begin{theorem} \label{result:theorem:stack_configuration}
A stack of maximum overhang with protruding block $p \in \{1,\dots,n\}$ satisfies the following two conditions:
\begin{itemize}
\item[a)] The stack is right-aligned at all blocks $i \geq p$.
\item[b)] The combined center of gravity of blocks $1$ through $p-1$ is aligned with the left edge of the protruding block $p$.
\end{itemize}
\end{theorem}

The proof closely follows the ideas presented in~\cite{treeby2018further}, requiring only a few adjustments to adapt it to blocks of arbitrary mass.
\begin{proof} 
\textit{a)} Suppose the stack is not right-aligned at block $i \geq p$, as illustrated in Figure~\ref{figure:notation} on the left. Then the center of gravity of blocks $1$ through $i$ is at a distance $d > 0$ to the left of the right edge of block $i+1$. Let $\delta, \epsilon \in (0, d/2]$ such that $\delta m_{i+1} = \epsilon M_i$. Shift block $i+1$ to the left by~$\delta$, and move blocks $1$ through $i$, including $p$, to the right by~$\epsilon$. The way $\delta$ and $\epsilon$ were defined, the center of gravity of blocks $1$ through $i+1$ is not affected by this change. Thus, the stack remains balanced but has a greater overhang, contradicting optimality of the original stack.

\textit{b)} Suppose the center of gravity of blocks $1$ through $p-1$ is at a distance $d > 0$ to the right of the left edge of $p$. Let $\delta, \epsilon \in (0, d/2]$ such that $\delta m_p = \epsilon M_{p-1}$. This time, shift blocks $1$ through $p-1$ to the left by $\epsilon$ and move $p$ to the right by $\delta$. Again, the stack remains balanced and has a larger overhang than it did initially.
\end{proof}

\begin{figure}[ht] 
\centering
\includegraphics{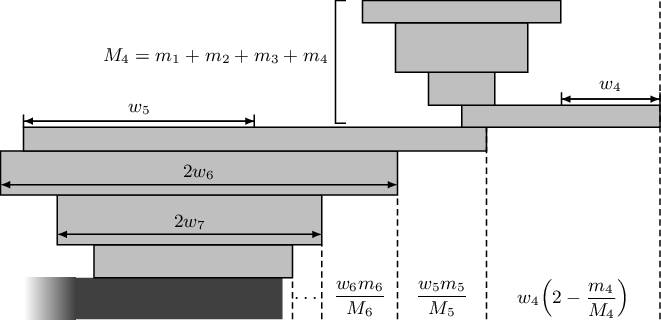}
\caption{The general structure of a stack with maximum overhang. The top three blocks serve as counterweights and have no direct contribution to the overhang. The right-aligned blocks incrementally form the overhang. Due to counterweights, the protruding block can contribute to the overhang with more than half its width.}
\label{figure:stack_configuration}
\end{figure}

By Theorem~\ref{result:theorem:stack_configuration}, and as depicted in Figure~\ref{figure:stack_configuration}, any optimal stack configuration divides the blocks into counterweights and right-aligned blocks, with the uppermost right-aligned block protruding. 
The counterweight blocks only contribute indirectly to the overhang by allowing the protruding block to extend further.
Each right-aligned block~$i$, on the other hand, directly generates an overhang equal to the distance between the combined center of gravity, $\bar{x}_{i}$, and the block's right edge, $x_i + w_i$. A straightforward calculation reveals that this distance is~$w_i m_i / M_i$, except for the protruding block $p$, which, due to the counterweights, contributes to the overhang with a distance of~$w_p ( 2 - m_p / M_p )$. Consequently, the maximum overhang of a stack with specified protruding block $p$ is
\begin{equation} \label{equation:overhang_fixed_configuration}
w_p  \bigg( 2 - \frac{m_p}{M_p} \bigg) + \sum_{i :\, i > p} \frac{w_i m_i}{M_i}, \text{\quad where \quad} M_i = \sum_{j :\, j \leq i} m_j. 
\end{equation}

Note that a higher counterweight increases the contribution of the protruding block to the overhang, whereas, for the other right-aligned blocks, a higher weight upon them is detrimental. This can also be observed in the following example.

\subsection{Stacking Configurations.}

Consider the blocks~$a$ and~$b$ with $w_a = 1$, $m_a = k$, \mbox{$w_b = k$}, and $m_b = 1$, for some $k \geq 2$. 
By~\eqref{equation:overhang_fixed_configuration}, with block~$a$ on top and without counterbalancing, the maximum overhang amounts to 
\[ w_a \bigg( 2 - \frac{m_a}{m_a} \bigg) + w_b \frac{m_b}{m_a + m_b} = 1 + \frac{k}{k+1} = 2 - \frac{1}{k+1} . \]
On the other hand, reversing the stacking order, that is, placing the longer and lighter block~$b$ on top, allows an overhang of 
\[ w_b \bigg( 2 - \frac{m_b}{m_b} \bigg) + w_a \frac{m_a}{m_b + m_a} = k + \frac{k}{1+k} . \]
Observe that the overhang cannot exceed two for the first stacking order, while the latter generates more than $k/2$ times as much. 
Thus, even considering only two blocks and no counterbalancing, rearranging the stacking order can improve the result by an arbitrarily large factor.

\begin{figure}[ht] 
\centering
\begin{subfigure}[b]{.24 \textwidth}
\centering
\includegraphics{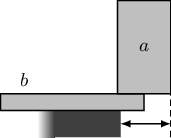}
\end{subfigure}
\hspace{1mm}%
\begin{subfigure}[b]{.24 \textwidth}
\centering
\includegraphics{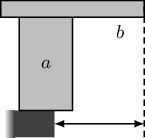}
\end{subfigure}
\hspace{-1mm}%
\begin{subfigure}[b]{.24 \textwidth}
\centering
\includegraphics{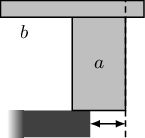}
\end{subfigure}
\hspace{1mm}%
\begin{subfigure}[b]{.24 \textwidth}
\centering
\includegraphics{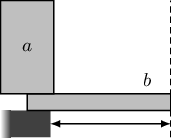}
\end{subfigure}
\caption{All four potentially optimal stack configurations without and with counterbalancing.}
\label{figure:example_rearranging}
\end{figure}

\newpage
If counterbalancing is permitted, two more stack configurations are available (see Figure~\ref{figure:example_rearranging}). With counterweight $m_b$, the overhang of block~$a$ is
\[ w_a \bigg( 2 - \frac{m_a}{m_b + m_a} \bigg) = 2 - \frac{k}{k+1}, \] 
which is even less than for the first stack. 
However, if block~$b$ protrudes and $a$ serves as a counterweight, the overhang is
\[ w_b \bigg( 2 - \frac{m_b}{m_a + m_b} \bigg) = k \bigg( 2 - \frac{1}{k+1} \bigg) = 2k - \frac{k}{k+1}, \]
which, for $k$ toward infinity, is up to twice as much as without counterbalancing. \\

This example demonstrates the power of rearranging the blocks, motivating the search for an optimal stack configuration, both with and without counterbalancing. 
In the following sections, we will analyze whether and how good stack configurations can be found, starting with a few simple stacking rules.

\subsection{Simple Stacking Rules.}

Notice that in the third stack in Figure~\ref{figure:example_rearranging}, a wide and light counterweight block extends farther than the short and heavy block that was supposed to protrude. Such a stack, of course, cannot be optimal. 
In fact, this stack configuration could have been ruled out from the outset because short and heavy blocks should never be used as the protruding block.
In this section, we formalize this and other simple stacking rules.

To follow the arguments below, it is crucial to understand how the overhang of a stack is formed and how it changes when, for example, two blocks are swapped. 
Recall from~\eqref{equation:overhang_fixed_configuration} that the contribution of a right-aligned block, apart from its own parameters, depends only on the aggregated mass of the blocks on top of it. 
Hence, when two blocks are swapped, this change does not affect the overhang generated by all blocks above and below. 
Moreover, remember that, for counterweight blocks, only their total mass is relevant for calculating the overhang.

\begin{lemma} \label{result:lemma:condition_stacking_order}
Consider an optimal stack configuration and let $a,b \in B$ be two right-aligned blocks, with block $a$ directly on top of block $b$. Then
\[ \frac{w_a}{M + m_a} \geq \frac{w_b}{M + m_b}, \]
where $M$ is the total mass of the blocks on top of block $a$.
\end{lemma}

\begin{proof}
Assume the contrary. According to~\eqref{equation:overhang_fixed_configuration}, swapping blocks $a$ and $b$ leads to the following change in the overhang:
\begin{align*}
\frac{w_b m_b}{M+m_b} + \frac{w_a m_a}{M+m_b+m_a} - \Big( \frac{w_a m_a}{M+m_a} + \frac{w_b m_b}{M+m_a+m_b} \Big) \, .
\end{align*}
By bringing all four fractions to a common denominator and 
suitably grouping the resulting terms, we find that this is equal to
\begin{align*}
&\frac{w_b m_b m_a}{(M+m_b)(M+m_a+m_b)} - \frac{w_a m_a m_b}{(M+m_a)(M+m_b+m_a)} \\
&= \frac{m_b m_a}{M + m_b + m_a} \cdot \Big( \frac{w_b}{M + m_b} - \frac{w_a}{M + m_a} \Big) 
> 0.
\end{align*}
This contradicts the optimality of the original stack configuration.
\end{proof}

\begin{remark}
Note that this condition is not sufficient for finding the best stacking order of right-aligned blocks, as the following counterexample illustrates. Let \mbox{$w_1 = 11$}, \mbox{$w_2 = 21$}, \mbox{$w_3 = 33$}, and $m_1 = 1$, $m_2 = 2$, $m_3 = 4$ be the parameters of blocks $1$, $2$, and $3$ of a fully right-aligned stack. Then the stacking order $1,2,3$ fulfills the condition, but a strictly better solution is given by the sequence $2,3,1$, as shown in Figure~\ref{figure:remark}. \vspace{-1mm}%
\begin{figure}[ht] 
\centering
\begin{subfigure}[b]{.32 \textwidth}
\centering
\includegraphics{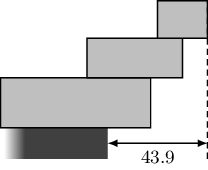}
\caption{}
\end{subfigure}
\begin{subfigure}[b]{.32 \textwidth}
\centering
\includegraphics{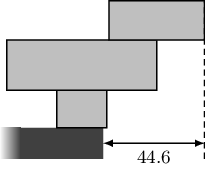}
\caption{}
\end{subfigure}
\begin{subfigure}[b]{.32 \textwidth}
\centering
\includegraphics{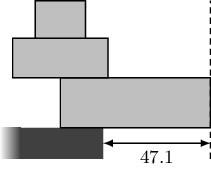}
\caption{}
\end{subfigure}
\caption{Both (a) and (b) satisfy the condition of Lemma~\ref{result:lemma:condition_stacking_order}, but (b) has a slightly larger overhang. If counterbalancing is permitted, stack configuration (c) is optimal.}
\label{figure:remark}
\end{figure}
\end{remark}

\begin{corollary} \label{result:corollary:special_condition_stacking_order}
Let $a, b \in B$ be two right-aligned blocks in an optimal stack configuration. If $w_a > w_b$ and $w_a / m_a > w_b / m_b$, block $a$ cannot be placed directly underneath~$b$. In particular, this is satisfied if $w_a > w_b$ and $m_a \leq m_b$.
\end{corollary}

\begin{proof}
\[ \frac{w_a}{M + m_a} - \frac{w_b}{M + m_b} = \frac{M (w_a - w_b) + (w_a m_b - w_b m_a) }{(M + m_a)(M + m_b)} > 0\]
for any $M \geq 0$. The claim follows from Lemma~\ref{result:lemma:condition_stacking_order}.
\end{proof}

Extending Corollary~\ref{result:corollary:special_condition_stacking_order} to the case $w_a / m_a = w_b / m_b$, it is easy to see that right-aligned blocks whose mass is proportional to their width can always be ordered by their width, with the widest on top, to obtain a stack of maximum overhang. 
As mentioned before, this was proven by Treeby~\cite{treeby2018further} and can be seen for both stacks in Figure~\ref{figure:treeby_blocks}.

In a more general sense, Corollary~\ref{result:corollary:special_condition_stacking_order} confirms that for right-aligned blocks, wider and lighter blocks should be placed above the shorter and heavier ones.
In particular, if one block is wider and at least as light as all the others, Corollary~\ref{result:corollary:special_condition_stacking_order} implies that this block cannot be placed beneath any right-aligned block. In this case, we can even prove that it cannot be used as a counterweight either. Hence, the block must protrude in order for the stack to achieve the best possible overhang.
\begin{lemma} \label{result:lemma:special_condition_protruding}
Let $b_\ast \in B$ be a block with $w_{b_\ast} > w_b$ and $m_{b_\ast} \leq m_b$ for all $b \in B \setminus \{ b_\ast \}$. Then, $b_\ast$ protrudes in any optimal stack configuration. 
\end{lemma}

\begin{proof}
Consider a stack configuration where $b_\ast$ does not protrude. By Corollary~\ref{result:corollary:special_condition_stacking_order}, $b_\ast$~cannot be a non-protruding right-aligned block. Therefore, $b_\ast$~must serve as a counterweight block above the protruding block $p \in B \setminus \{b_\ast\}$. Let $M$ be the total mass of the other counterweight blocks. 
Now consider the stack configuration with $b_\ast$ and~$p$ swapped and $b_\ast$ protruding. 
Keep in mind that the overhang generated by the blocks below the protruding block is not affected by this change (cf.~\eqref{equation:overhang_fixed_configuration}). 
Since $m_{b\ast} \leq m_p$, the new stack configuration achieves a larger overhang than the original one:
\begin{align*}
& w_{b_\ast} \bigg( 2 - \frac{m_{b_\ast}}{M + m_p + m_{b_\ast}} \bigg) 
- w_p \bigg( 2 - \frac{m_p}{M + m_{b_\ast} + m_p} \bigg) \\
&\geq \big( \underbrace{w_{b_\ast} - w_p}_{> 0} \big) \cdot \bigg( \underbrace{2 - \frac{m_p}{M + m_{b_\ast} + m_p}}_{\geq 1} \bigg) > 0.
\end{align*} 

Thus, in an optimal stack configuration, $b_\ast$ is neither a counterweight block nor a non-protruding right-aligned block. 
\end{proof}

Additional insights similar to this result can be obtained. However, in the next section, we demonstrate that finding an optimal stack configuration is provably difficult despite all these insights into the Block-Stacking Problem.

\section{NP-Hardness.} \label{section:np_hardness}

For a fixed stacking order and a designated protruding block~$p$, the achievable overhang is given in closed form by~\eqref{equation:overhang_fixed_configuration}. 
The goal of the Block-Stacking Problem is to identify one of these stack configurations that allows for the largest possible overhang.
Based on~\eqref{equation:overhang_fixed_configuration}, using a bijective function $\pi \colon B \to \{1,\dots,\lvert B \rvert\}$ to represent the stacking order (top to bottom), this can be concisely formulated as 
\begin{equation} \label{equation:overhang_flexible_configuration}
\max_{p \in B}\ \max_\pi\ \  w_p  \bigg( 2 - \frac{m_p}{M_p} \bigg) + \sum_{i :\, \pi(i) > \pi(p)} \frac{w_i m_i}{M_i}, \text{\ \ \ where \ \ \ } M_i = \sum_{j:\, \pi(j) \leq \pi(i)} m_j. 
\end{equation}
Since there are $n \cdot n!$, i.e., exponentially many possible stack configurations, computing the potential overhang for all of them is a highly inefficient procedure that would exceed any available computing power even for moderately small~$n$.
The results discussed in the previous section may reduce the time required to find an optimal solution to some extent. However, these improvements are not sufficient to devise an efficient algorithm. 

In complexity theory, an algorithm is considered efficient if, for any possible input, the number of computational steps required is polynomial in the size of the input parameters. 
An algorithm that satisfies this condition is also called a polynomial-time algorithm. 
No polynomial-time algorithm has ever been found for many interesting problems, such as the famous Traveling Salesperson Problem. 
This led to the creation of the complexity class of so-called NP-hard problems. If any one of them can be solved in polynomial time, then all of them can (in which case $\text{P} = \text{NP}$, which is considered unlikely).

In this section, we demonstrate that Block Stacking is NP-hard. 
To do this, we show that any efficient algorithm for the Block-Stacking Problem could be utilized to efficiently solve some problem already known to be NP-hard. We will 
show that solving the Block-Stacking Problem in polynomial time would enable us to solve the  NP-hard Partition Problem in polynomial time. For a detailed introduction to complexity theory and a proof of NP-hardness of the Partition Problem, see, e.g.,~\cite{garey1979computers}.

\begin{definition}[Partition Problem]
Given a set of positive integers $A = \{a_1,\dots,a_n\}$, decide whether the elements can be partitioned into two disjoint subsets $A_1 \cup A_2 = A$, such that the sum of elements in both subsets is equal, i.e., $\sum_{a_i \in A_1} a_i = \sum_{a_i \in A_2} a_i$.
\end{definition}

If such a partition exists, we call it a \textit{perfect partition}. Without loss of generality, we can assume $\sum_{a_i \in A} a_i$ to be even. Accordingly, let $T = \frac{1}{2} \sum_{a_i \in A} a_i \geq 1$ denote the targeted split. 

\subsection{Solving the Partition Problem via the Block-Stacking Problem.}

Given an instance $A = \{a_1,\dots,a_n\}$ of the Partition Problem, consider the following set $B$ of blocks:
\begin{itemize}
\item for $i = 1,\dots,n$, let $i \in B$ be a block of mass $m_i = a_i$ and half-width $w_i = 1$,
\item let $b^\bullet \in B$ be a block of mass $m_\bullet = 1 $ and half-width $ w_\bullet = \big(2T + \frac{5}{4} \big)^5$, and
\item let $b^\ast \in B$ be a block of mass $m_\ast = 1/4 $ and half-width 
\[ w_\ast = \frac{w_\bullet m_\bullet}{m_\ast} \frac{(T+m_\ast)^2}{(T+m_\ast+m_\bullet)^2} = 4 w_\bullet \bigg( 1 - \frac{1}{T + \frac{5}{4}} \bigg)^2. \]
\end{itemize}

\begin{figure}[ht] 
\centering
\includegraphics{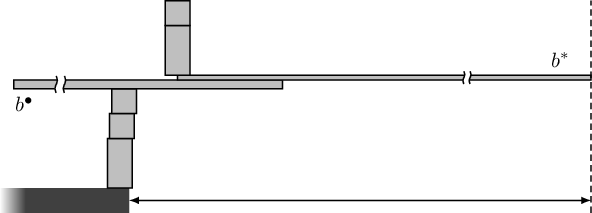}
\caption{Schematic structure of a $\bullet \ast$-protruding stack. $b^\bullet$ and $b^\ast$ together partition the remaining blocks into counterweights and right-aligned blocks.}
\label{figure:bullet_star_protruding}
\end{figure}

The goal is to prove that solving the Block-Stacking Problem for this set of blocks automatically reveals the answer to the corresponding instance of the Partition Problem.
More specifically, we want to show that in any optimal stack configuration, the block $b^\ast$ protrudes and $b^\bullet$ is placed directly underneath. 
As a result, together they partition the blocks $1,\dots,n$ into counterweights and right-aligned blocks, as illustrated in Figure~\ref{figure:bullet_star_protruding}.
We refer to a stack of this structure as \textit{$\bullet \ast$-protruding}. 
The weight and mass of the blocks were conveniently chosen to satisfy all the conditions of Lemma~\ref{result:lemma:bullet_star_protruding} and~\ref{result:lemma:omin_omax}, as well as to simplify the corresponding proofs as much as possible. While having two auxiliary blocks offered a certain degree of freedom in the choice of these parameters, note that this construction was not possible using a single auxiliary block.
Before showing that this helps in finding a perfect partition, we first prove that any optimal solution to the given instance of the Block-Stacking Problem is indeed $\bullet \ast$-protruding.

\begin{lemma} \label{result:lemma:bullet_star_protruding}
Any optimal stack configuration is $\bullet \ast$-protruding.
\end{lemma}

\begin{proof}
We will prove this claim in three steps. In any optimal stack configuration
\begin{itemize}
\item[a)] $b^\ast$ protrudes,
\item[b)] $b^\bullet$ is a right-aligned block, and
\item[c)] $b^\bullet$ is placed directly underneath $b^\ast$.
\end{itemize}  

\textit{a)} By Lemma~\ref{result:lemma:special_condition_protruding}, it suffices to show $w_\ast > w_\bullet$, $w_\ast > w_i$, and $m_\ast \leq m_\bullet$, $m_\ast \leq m_i$ for all $i \in B \setminus \{b^\ast, b^\bullet\}$. The latter is trivial, as $m_\ast = 1/4 < 1 = m_\bullet \leq m_i$. Moreover, 
\begin{align*}
w_\ast &= 4 w_\bullet \bigg( 1 - \frac{1}{T + \frac{5}{4}} \bigg)^2 \geq 4 w_\bullet \bigg( 1 - \frac{1}{1 + \frac{5}{4}} \bigg)^2 > w_\bullet = \big(2T + \tfrac{5}{4} \big)^5 > 1 = w_i.
\end{align*}

\textit{b)} Suppose $b^\bullet$ serves as a counterweight. First, consider the case where $b^\bullet$ is the only counterweight block. This means, there is a right-aligned block $i \in B \setminus \{b^\ast, b^\bullet\}$ directly below $b^\ast$. Swapping $b^\bullet$ and $i$ increases the overhang by
\begin{equation*}
w_\ast \Big( \underbrace{\frac{m_i}{m_i + m_\ast} - \frac{m_\bullet}{m_\bullet + m_\ast}}_{\geq 0 \text{ as } m_i \geq m_\bullet} \Big) + \frac{w_\bullet m_\bullet - w_i m_i}{m_\ast + m_\bullet + m_i} 
\geq \frac{\big(2T + \frac{5}{4} \big)^5 - 2T}{m_\ast + m_\bullet + m_i} > 0.
\end{equation*}

Now consider the case where the total mass of the counterweight blocks, $b^\bullet$ excluded, is $M \geq 1$. 
Moving $b^\bullet$ directly under $b^\ast$ increases the overhang by 
\begin{align*}
& w_\ast \Big( 2 - \frac{m_\ast}{M+m_\ast} \Big) + \frac{w_\bullet m_\bullet}{M+m_\ast+m_\bullet} - w_\ast \Big( 2 - \frac{m_\ast}{M+m_\ast+m_\bullet} \Big) \\
& = \frac{w_\bullet m_\bullet}{M+m_\ast+m_\bullet} + w_\ast \Big( \frac{m_\ast}{M+m_\ast+m_\bullet} - \frac{m_\ast}{M+m_\ast} \Big) \\
& = \frac{m_\bullet}{M+m_\ast+m_\bullet} \cdot \Big( w_\bullet - \frac{w_\ast m_\ast}{M + m_\ast} \Big),
\end{align*}
which is strictly greater than zero, since $ M \geq 1$ and 
\begin{align*}
w_\ast m_\ast = w_\bullet \Big( 1 - \frac{m_\bullet}{T +m_\ast+m_\bullet} \Big)^2 < w_\bullet.
\end{align*}

\textit{c)} As $w_\bullet > w_i$ and $m_\bullet \leq m_i$ for all $i \in B \setminus \{b^\ast, b^\bullet\}$, it immediately follows from Corollary~\ref{result:corollary:special_condition_stacking_order} that $b^\bullet$ cannot be placed beneath any of those blocks in an optimal solution. Since we have already seen that $b^\bullet$ is a right-aligned block, it must be placed directly underneath the protruding block $b^\ast$. 
Therefore, any optimal stack configuration is $\bullet \ast$-protruding.
\end{proof}

By Lemma~\ref{result:lemma:bullet_star_protruding}, for any stack configuration with maximum overhang, $b^\bullet$ and $b^\ast$ partition the blocks $1,\dots,n$ into counterweights and right-aligned blocks.
Given any set of numbers $a_1,\dots,a_n$, we now want to show that if a perfect partition is possible, it is best to partition the corresponding blocks $1,\dots,n$ such that both the counterweights and right-aligned blocks each have total mass $T$.
To this end, we assume the total sum of elements $2T = \sum_{a_i \in A} a_i$ to be fixed and introduce the functions $O_{\min}(C)$ and $O_{\max}(C)$ as the minimum and maximum overhang achievable with a particular counterweight mass $C \in \{0,\dots,2T\}$. 
In order to find explicit expressions for $O_{\min}(C)$ and $O_{\max}(C)$, we prove the following lemma.

\begin{lemma} \label{result:lemma:splitting_blocks}
Replacing a block $i \in B \setminus \{b^\ast, b^\bullet\}$ by two blocks ${i_1}$, ${i_2}$ of the same half-width $w_{i_1}=w_{i_2}=w_i = 1$ and with masses $m_{i_1} + m_{i_2} = m_i$, the overhang does not decrease.
\end{lemma}

\begin{proof}
If $i$ serves as counterweight, the overhang remains unchanged. Otherwise, let $M$ be the mass of all blocks above $i$. Replacing $i$ by $i_1$ and $i_2$, where $i_1$ is on top of~$i_2$, the overhang of the stack does not decrease:
\begin{align*}
\frac{w_{i_1} m_{i_1}}{M + \underbrace{m_{i_1}}_{\leq m_i}} + \frac{w_{i_2} m_{i_2}}{M + \underbrace{m_{i_1} + m_{i_2}}_{=m_i}} - \frac{w_i m_i}{M + m_i} 
\geq w_i\ \frac{m_{i_1} + m_{i_2} - m_i}{M+m_i} = 0. \\[-\normalbaselineskip]\tag*{\qedhere}
\end{align*} 
\end{proof}

By Lemma~\ref{result:lemma:splitting_blocks}, a $\bullet \ast$-protruding stack has maximum overhang if the right-aligned blocks are split into as many blocks as possible.
Given that $a_1, \dots, a_n$ can only be positive integers, the overhang of the corresponding stack can be at most the overhang of a $\bullet \ast$-protruding stack whose regular blocks all have mass $m_i = 1$.
Consequently, for counterweight~$C$,
\begin{equation*}
\begin{gathered}
O_{\max}(C) = w_\ast \bigg(2-\frac{m_\ast}{C+m_\ast}\bigg) + \frac{w_\bullet m_\bullet}{C+m_\ast+m_\bullet} + H_C,  \\
\text{where\quad} H_C = \sum_{i=1}^{2T - C} \frac{1}{C + m_\ast + m_\bullet + i} \enspace . 
\end{gathered}
\end{equation*}
Conversely, the minimum overhang is achieved when there is only a single block of mass $2T-C$ under the $\bullet \ast$-protruding structure, as shown in Figure~\ref{figure:omin_omax}. Therefore,
\[ O_{\min}(C) = w_\ast \bigg(2-\frac{m_\ast}{C+m_\ast}\bigg) + \frac{w_\bullet m_\bullet}{C+m_\ast+m_\bullet} + \frac{2T-C}{2T + m_\ast + m_\bullet} \enspace . \]

\begin{figure}[ht] 
\centering
\begin{subfigure}[b]{.48 \textwidth}
\centering
\includegraphics{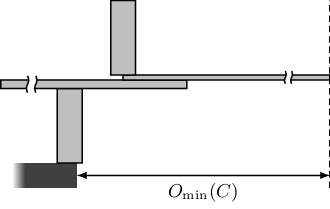}
\end{subfigure}
\begin{subfigure}[b]{.48 \textwidth}
\centering
\includegraphics{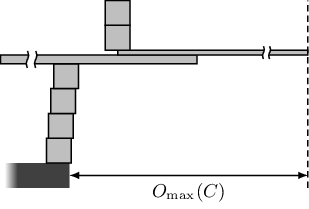}
\end{subfigure}
\caption{Structure of the $\bullet \ast$-protruding stacks with minimum and maximum overhang. By Lemma~\ref{result:lemma:omin_omax}, the overhang of the stack on the left is greater than that of any stack without a perfect partition, such as the one on the right.}
\label{figure:omin_omax}
\end{figure}

To demonstrate that a stack configuration with maximum overhang relies on a perfect partition, we prove that the worst possible stack configuration with counterweight~$T$ still has a greater overhang than any stack configuration with a different split. 

\begin{lemma} \label{result:lemma:omin_omax}
$O_{\min}(T) > O_{\max}(T \pm K)$ for all $K = 1,\dots,T$.
\end{lemma}

\begin{proof}
It suffices to show
\begin{equation*} 
\begin{gathered}
 w_\ast \bigg( 2 - \frac{m_\ast}{T+m_\ast} \bigg) + \frac{w_\bullet m_\bullet}{T+m_\ast+m_\bullet} \\
 > w_\ast \bigg( 2 - \frac{m_\ast}{T \pm K+m_\ast} \bigg) + \frac{w_\bullet m_\bullet}{T \pm K+m_\ast+m_\bullet} + H_{T \pm K},
\end{gathered} 
\end{equation*} 
or, equivalently,
\begin{equation*}
\mp \frac{Kw_\ast m_\ast}{(T+m_\ast)(T\pm K+m_\ast)} \pm \frac{K w_\bullet m_\bullet}{(T+m_\ast+m_\bullet)(T \pm K+m_\ast+m_\bullet)} > H_{T \pm K} \enspace .
\end{equation*}
By definition of $w_\ast$,
\begin{align*}
& \mp \frac{Kw_\ast m_\ast}{(T+m_\ast)(T\pm K+m_\ast)} \pm \frac{K w_\bullet m_\bullet}{(T+m_\ast+m_\bullet)(T \pm K+m_\ast+m_\bullet)} \\
& = \mp \frac{K w_\bullet m_\bullet}{(T+m_\ast + m_\bullet)^2} \bigg( \frac{T+m_\ast}{T \pm K + m_\ast} - \frac{T+m_\ast +m_\bullet}{T \pm K + m_\ast + m_\bullet} \bigg) \\
& = \frac{K^2 w_\bullet m_\bullet^2}{(T+m_\ast+m_\bullet)^2 (T \pm K + m_\ast + m_\bullet) (T \pm K + m_\ast)} \enspace .  
\end{align*}
As $1 \leq K \leq T$, $m_\bullet = 1$, and $w_\bullet = (2T + m_\ast + m_\bullet)^5$, this is greater than $2T$, which in turn is greater than $H_{T \pm K}$.
\end{proof}

\begin{theorem} \label{result:theorem:np_hardness}
Determining an optimal stack configuration for the Block-Stacking Problem with counterbalancing is NP-hard.
\end{theorem}

\begin{proof}
Given a set of positive integers $A = \{a_1,\dots,a_n\}$ for the Partition Problem, consider the instance of the Block-Stacking Problem as defined and analyzed above. Note that this instance and all its parameters are of polynomial size. By Lemma~\ref{result:lemma:bullet_star_protruding} and~\ref{result:lemma:omin_omax} it is clear that a perfect partition of $A$ exists if and only if all optimal solutions to the Block-Stacking Problem are $\bullet \ast$-protruding and utilize a counterweight of exactly~$T$. Given a stack configuration that optimally solves the Block-Stacking Problem, its counterweight can be efficiently calculated. Therefore, Block Stacking is NP-hard.
\end{proof}

\section{Block Stacking without Counterbalancing.} \label{section:stacks_without_counterbalancing}

In the previous section, we found that Block Stacking is, in general, NP-hard, making it effectively intractable to efficiently identify an optimal stack configuration. However, what is the situation if only stacks without counterbalancing, i.e., fully right-aligned stacks, as in Figure~\ref{figure:general_blocks} on the right, are allowed? For blocks whose mass is proportional to their width, Treeby~\cite{treeby2018further} showed that without counterbalancing, the Block-Stacking Problem is significantly easier to analyze and can be solved by simply sorting the blocks by their width (cf.~Lemma~4.2 in \cite{treeby2018further}, Corollary~\ref{result:corollary:special_condition_stacking_order}, and Figure~\ref{figure:treeby_blocks}). No such simple algorithm seems to be possible for blocks of arbitrary mass, and it remains an open question as to whether or not an exact solution can be found efficiently. Nevertheless, there is a lot to be said about this variant. Ultimately, the following discussion will even help to find an approximation algorithm for the general Block-Stacking Problem (i.e., with counterbalancing).\\

When a stack is fully right-aligned, the protruding block is always the one on top. Thus, only the stacking order needs to be optimized, and~\eqref{equation:overhang_flexible_configuration} can be simplified to
\begin{equation} \label{equation:bsp_no_counterbalancing}
 \max_\pi \, \sum_{i \in B} \frac{w_i m_i}{M_i} \text{,\quad where\quad } M_i = \sum_{j :\, \pi(j) \leq \pi(i)} m_j. 
\end{equation}

This version of the Block-Stacking Problem has not been studied before. 
It turns out to be structurally similar to two problems with entirely different backgrounds. The following two sections will prove a close connection between Block Stacking, Airplane Refueling, and Robust Appointment Scheduling.

\section{The Airplane Refueling Problem.} \label{section:ar}

The Airplane Refueling Problem, also known as the Jeep Caravan, the Jeep Convoy, or the $n$-Vehicle Exploration Problem, originates from the following mathematical puzzle (cf.~\cite{franklin1960range,vasquez2015airplane, woeginger2010airplane}). Suppose some cargo needs to be carried as far as possible, farther than a single aircraft could manage. Instead, a fleet of aircraft is deployed to accomplish the task. 
Fuel can be exchanged instantaneously and in flight between the airplanes. When a plane runs out of fuel or has passed all of it to the other airplanes, it drops out and no longer contributes to the fuel consumption. Thus, given the fuel consumption rate and tank volume of each aircraft, the Airplane Refueling Problem aims to determine a dropout sequence that maximizes the distance the last remaining aircraft can travel. 
Alternatively, the Jeep Caravan Problem~\cite{phipps1947jeep} and the $n$-Vehicle Exploration Problem~\cite{zhang2021novel} consider the task of crossing a desert using a convoy of jeeps.
Mathematically, however, they are identical to the Airplane Refueling Problem. 
It should be noted that the Jeep Problem also refers to some variants where only a single vehicle is considered (see, e.g.,~\cite{fine1947jeep, gale1970jeep, hausrath1995gale, korf2022jeep, phipps1947jeep}). 
For clarity, we will stick to the context of Airplane Refueling for the remainder of this paper. \\

Formally, let $A = \{1,\dots,n\}$ be a set of airplanes, each specified by its \textit{tank volume}~$v_i \geq 0$ and \textit{fuel consumption rate}~$c_i > 0$. Of course, no fuel should be left unused, but apart from that, every aircraft should be eliminated as soon as possible to reduce the overall fuel consumption. 
Correspondingly, it is optimal to stop an airplane as soon as the continuing aircraft can carry all the remaining fuel on their own.
Following this procedure, the fleet has, under the dropout sequence $\sigma \colon \{1,\dots,n\} \to A$, a range of
\[ \frac{v_{\sigma(1)}}{c_{\sigma(1)} + \cdots + c_{\sigma(n)}} + \cdots + \frac{v_{\sigma(n-1)}}{c_{\sigma(n-1)} + c_{\sigma(n)}} + \frac{v_{\sigma(n)}}{c_{\sigma(n)}} = \sum_{i = 1}^{n} \frac{v_{\sigma(i)}}{\sum_{j=i}^{n} c_{\sigma(j)}} \enspace. \]
Define $\pi \colon A \to \{1,\dots,n\}$ to be the inverse of~$\sigma$ and observe that
\[ \sum_{i = 1}^{n} \frac{v_{\sigma(i)}}{\sum_{j=i}^{n} c_{\sigma(j)}} = \sum_{i \in A} \frac{v_i}{C_i}, \quad \text{where} \quad C_i \ = \ \quad \ \sum_{\mathclap{ j :\, \sigma^{-1}(j) \geq \sigma^{-1}(i)}} \quad\ c_j \ = \sum_{j :\, \pi(j) \geq \pi(i)} c_j. \]
Hence, maximizing the range of the fleet comes down to solving 
\begin{equation} \label{equation:ar_original}
\max_\pi \, \sum_{i \in A} \frac{v_i}{C_i}, \quad \text{where} \quad C_i = \sum_{j:\, \pi(j) \geq \pi(i)} c_j.
\end{equation}
Alternatively, with $\pi^\prime \colon A \to \{1,\dots,n\}$ defined as the reverse of $\pi$, that is, \linebreak $\pi^\prime(i) =  n + \nolinebreak 1 - \pi(i)$ for all $i \in A$, the Airplane Refueling Problem is
\begin{equation} \label{equation:ar_scheduling}
\max_{\pi^\prime} \, \sum_{i \in A} \frac{v_i}{C_i}, \quad \text{where} \quad C_i = \sum_{j :\, \pi^\prime(j) \leq \pi^\prime(i)} c_j.
\end{equation}

Comparing the last formulation to Block Stacking without counterbalancing, as given in~\eqref{equation:bsp_no_counterbalancing}, an interesting similarity can be observed. 
Indeed, by identifying $A = B$, $c_i = m_i$ and $v_i = w_i m_i$ for all $i \in A$, solving~\eqref{equation:ar_scheduling} is equivalent to solving~\eqref{equation:bsp_no_counterbalancing}.
Airplane Refueling, in turn, can be modeled as an instance of the Block-Stacking Problem by setting $B = A$, $m_i = c_i$, and $w_i = v_i / c_i$ for all $i \in B$.
Thus, any optimal stacking order for the Block-Stacking Problem corresponds to an optimal dropout sequence for the respective instance of the Airplane Refueling Problem, and vice versa.
\begin{theorem} \label{result:theorem:ar_bsp_equivalent}
The Block-Stacking Problem without counterbalancing is equivalent to the Airplane Refueling Problem.
\end{theorem}

As a result, both problems have the same computational complexity. In addition, any algorithm, including approximation algorithms, known for one problem can be immediately applied to the other. 
In particular, recall that the Airplane Refueling Problem permits a polynomial-time approximation scheme~\cite{gamzu2019polynomial}. That is, for every $\epsilon>0$, there exists a $(1+\epsilon)$-approximation algorithm.
\begin{corollary} \label{result:corollary:1_plus_e_approximation_no_counterbalancing}
There exists a polynomial-time approximation scheme for the Block-Stacking Problem without counterbalancing.
\end{corollary}

\section{Robust Appointment Scheduling.} \label{section:ras}
We now look at Robust Appointment Scheduling, a seemingly unrelated application-oriented optimization problem. Nonetheless, we will establish a connection to the Airplane Refueling Problem. Hence, Robust Appointment Scheduling and Airplane Refueling and Block Stacking form a conspicuous group of closely related problems. 

In the service industry, especially in high-cost sectors like healthcare, scheduling appointments as efficiently as possible can be a significant advantage. However, since the time required for the medical tasks to be performed can usually only be estimated, any schedule designed in advance will inevitably result in some of them being completed before or after their allotted time. Even though unavoidable, both situations are unfavorable. If a task is finished early, personnel, equipment, and facilities remain unused until the next task is due to start, resulting in underutilization costs, as this time could have been allocated to another task. On the other hand, a task that finishes late not only incurs costs due to its own tardiness, which we will refer to as the task’s overage costs, but also prevents the next task from starting at its scheduled time. Thus, a schedule that is too tight can cause considerable costs due to an accumulation of delays. The problem of Appointment Scheduling is finding a schedule that optimally balances the costs arising from possible underutilization and delays. Note that a crucial part of reducing costs is to adjust the order in which the tasks are processed. Because the uncertainty of a schedule increases with the number of previous tasks, those with high associated costs should be processed at the beginning. At the same time, placing tasks with less uncertainty at the beginning increases the reliability of the entire schedule, reducing overall costs. The complexity of the problem arises when the tasks at hand cannot be arranged so that both principles are fulfilled simultaneously.\\

In the following, we consider the model presented by Mittal et al.~\cite{mittal2014robust} based on robust optimization.
For a set of jobs (i.e., tasks) $J = \{1,\dots,n\}$, each characterized by an interval of possible processing times~$[\underline{p}_i, \overline{p}_i]$, a per-unit \textit{overage cost}~$o_i > 0$, and a common per-unit \textit{underutilization cost}~$u > 0$, the goal is to find a schedule consisting of a processing order $\pi \colon J \to \{1,\dots,n\}$ and a time period~$t_i$ allocated to each job~$i \in J$ that minimizes the cost of a worst-case scenario. Extensions and variants of this model have been explored in \cite{bauerhenne2024robust,schulz2019robust}.

As proved in~\cite{mittal2014robust}, for a fixed processing order~$\pi$, the optimal values for $t_1,\dots,t_n$ are given by a weighted average of the lower and upper bounds on the processing time of the respective jobs.
With $\Delta_i = \overline{p}_i - \underline{p}_i$, they are 
\[
t_i = \frac{u \underline{p}_i + O_i \overline{p}_i}{u + O_i} = \underline{p}_i + \Delta_i \frac{O_i}{u + O_i} , \quad \text{where} \quad
O_i = \sum_{j :\, \pi(j) \geq \pi(i)} o_j. 
\]
Moreover, the total cost of a worst-case scenario can also be expressed in closed form. As a result, finding the best processing order becomes the following optimization problem:
\begin{equation*}
\min_\pi \, \sum_{i \in J} \Delta_i \frac{u O_i}{u + O_i}, \quad \text{where} \quad O_i = \sum_{j :\, \pi(j) \geq \pi(i)} o_j .
\end{equation*}

\noindent Note that regardless of the processing order~$\pi$, it holds
\[ \sum_{i\in J} \Delta_i \frac{u O_i}{u + O_i} + u^2 \sum_{i\in J} \frac{\Delta_i}{u + O_i} = \sum_{i\in J} \Delta_i \frac{uO_i + u^2}{u + O_i} = u \sum_{i\in J} \Delta_i. \]
As the right-hand side is independent of $\pi$, minimizing the first sum on the left is equivalent to maximizing the second. Moreover, as $u > 0$, we have:
\begin{lemma} \label{result:lemma:ras_sict_equivalent}
A processing order~$\pi$ is optimal for Robust Appointment Scheduling if and only if it is optimal for 
\begin{equation} \label{equation:sict}
\max_\pi \, \sum_{i\in J} \frac{\Delta_i}{u + O_i}, \quad \text{where} \quad O_i = \sum_{j :\, \pi(j)  \geq \pi(i)} o_j.
\end{equation}
\end{lemma}

Comparing~\eqref{equation:sict} with the Airplane Refueling Problem, as given in~\eqref{equation:ar_original},
both problems would be equivalent, if it were not for the shift~$u > 0$ in the denominator of~\eqref{equation:sict}. 
We still exploit this similarity to show that the Airplane Refueling Problem can be used for solving~\eqref{equation:sict}.

The idea of the transformation is to identify jobs with uncertainty range~$\Delta_i$ and overage costs~$o_i$ with airplanes that have tank volume $v_i = \Delta_i$ and fuel consumption rate $c_i = o_i$. It then remains to incorporate the shift~$u > 0$ into the aggregated fuel consumption rates~$C_i$ by adding an auxiliary airplane~$a^\ast$ with fuel consumption rate~$c_\ast = u$. 
To ensure $u$ is added to all $C_i$, the auxiliary airplane must be dropped last in any optimal solution, which can be guaranteed by setting its tank volume $v_\ast$ sufficiently large.
To demonstrate this more formally, we need the following necessary optimality condition for the Airplane Refueling Problem. 
Using the connection to the Block-Stacking Problem developed in Section~\ref{section:ar}, it is easy to see that this is essentially a reformulation of Lemma~\ref{result:lemma:condition_stacking_order}.

\begin{lemma} \label{result:lemma:condition_dropout_order}
Assume that the airplanes $A=\{1,\dots,n\}$ are indexed according to an optimal dropout order, i.e., $\pi(i) = i$ for all $i \in A$. Then, for the aggregated fuel consumption rates \mbox{$C_{n+1} := 0 < C_n < \cdots < C_2 < C_1$}, for all $i = 2,\dots,n$, it holds 
\[ f_i(C_{i+1}) \geq f_{i-1}(C_{i+1}), \quad \text{where} \quad f_i(x) = \frac{v_i}{c_i(x + c_i)} \enspace . \]
\end{lemma}

\begin{lemma} \label{result:lemma:airplane_dropped_last}
Let $v_i$, $c_i$ be the tank volumes and fuel consumption rates of the airplanes~$i \in A$ with $v_i \neq 0$ for at least one~$i \in A$. 
Given an additional airplane $a^\ast$ with fuel consumption rate $c_\ast > 0$, there exists a tank volume $v_\ast$ of polynomial size such that in any optimal dropout order the additional airplane is dropped last. 
In particular,
\[ v_\ast = \frac{c_\ast \max_j v_j}{c_\bullet^2}\ \bigg(c_\ast + \sum_{j \in A} c_j \bigg), \]
where $c_\bullet = \min\{c_\ast, c_j \mid j \in A\}$, fulfills these requirements.
\end{lemma}

\begin{proof} 
By the definitions of $v_\ast$ and $c_\bullet$, for all $i \in A$ and $x \in [0, \sum_{j \in A} c_j]$, it holds
\begin{equation*}
f_\ast(x) = \frac{v_\ast}{c_\ast (c_\ast + x)} = \frac{\max_j v_j}{c_\bullet^2} \underbrace{\frac{c_\ast (c_\ast + \sum_{j \in A} c_j) }{c_\ast (c_\ast + x)}}_{\geq 1} \overset{(\text{I})}{\geq} \frac{v_i}{c_i^2} \overset{(\text{II})}{\geq} \frac{v_i}{c_i (c_i + x)} = f_i(x).
\end{equation*}
In fact, we can even see that $f_\ast(x) > f_i(x)$, as $(\text{I})$ can hold with equality only if $v_i = \max_j v_j$ and $x = \sum_j c_j$, while $(\text{II})$ holds with equality if and only if $x = 0$ or $v_i = 0 < \max_j v_j$. 
Since $C_{a^\ast} - c_\ast$ must lie within $[0, \sum_{j \in A} c_j]$, Lemma~\ref{result:lemma:condition_dropout_order} implies that in any optimal solution, the additional airplane must be dropped last. 
\end{proof}

\begin{lemma} \label{result:lemma:sict_ar_reduction}
A polynomial-time algorithm for the Airplane Refueling Problem implies a polynomial-time algorithm for~\eqref{equation:sict}.
\end{lemma}

\begin{proof}
Consider an instance of~\eqref{equation:sict} given by $\Delta_i \geq 0$, $o_i > 0$ for all jobs $i \in J$, and the shift~$u > 0$. Without loss of generality, we can assume that $\Delta_i \neq 0$ for at least one $i \in J$, as otherwise solving~\eqref{equation:sict} is trivial. 
Let $A = J \cup \{a^\ast\}$ be a set of airplanes with the following specifications. For each $i \in J$, consider an aircraft with tank volume $v_i = \Delta_i$ and fuel consumption rate $c_i = o_i$. Furthermore, add the auxiliary airplane~$a^\ast$ with fuel consumption rate $c_\ast = u$ and tank volume $v_\ast$, as specified in Lemma~\ref{result:lemma:airplane_dropped_last}. Therefore, this aircraft has to be dropped last in any optimal solution to this instance of the Airplane Refueling Problem. 
Hence, any optimal dropout order $\pi^\ast \colon A \to \{1,\dots,n+1\}$ satisfies $\pi^\ast(a^\ast) = n+1$ and $\pi^\ast(i) = \pi(i)$ for all $i \in J$ for a permutation~$\pi \colon J \to \{1,\dots,n\}$. 
Let $C_i$ be the accumulated fuel consumption rates for the airplanes $i \in A$, and $O_i$ be the accumulated overage costs of the jobs $i\in J$ with respect to $\pi$. 
Then $O_i = C_i - c_\ast$ for all $i \in J$ and 
\[ \sum_{i \in A} \frac{v_i}{C_i} = \frac{v_\ast}{C_{a^\ast}} + \sum_{i\in J} \frac{v_i}{c_\ast + (C_i - c_\ast)} = \frac{v_\ast}{c_\ast} + \sum_{i\in J} \frac{\Delta_i}{u + O_i}. \]

Since $v_\ast / c_\ast$ is constant, the permutation~$\pi$ is optimal for the given instance of the Airplane Refueling Problem, \vspace{-3mm}%
\begin{equation*}
\max_\pi \, \sum_{i \in A} \frac{v_i}{C_i}, \vspace{-1mm}
\end{equation*} 
if and only if it is optimal for~\eqref{equation:sict}, i.e., for \ \vspace{-2mm}
\[ \max_\pi \, \sum_{i \in J} \frac{\Delta_i}{u + O_i}. \]
By Lemma~\ref{result:lemma:airplane_dropped_last}, $v_\ast$ is of polynomial size with respect to the input parameters. Hence, solving the given instance of the Airplane Refueling Problem via a polynomial-time algorithm constitutes a polynomial-time algorithm for~\eqref{equation:sict}.
\end{proof}

It is also possible to prove the reverse. 
In the proof of Lemma~\ref{result:lemma:sict_ar_reduction}, the shift~$u$ was modeled as an additional airplane.
Conversely, we can remove one airplane at a time and simulate its fuel consumption rate using the shift. 
Solving~\eqref{equation:sict} then provides the optimal permutation for the Airplane Refueling Problem, assuming that the removed airplane has to be dropped last.
The overall optimal dropout order can be identified by putting each airplane in this situation once and comparing the resulting objective values.
This proof idea can readily be formalized, but is omitted here due to its similarity to the proof of Lemma~\ref{result:lemma:sict_ar_reduction}. Nonetheless, taking into account Lemma~\ref{result:lemma:ras_sict_equivalent}, overall, we can state the following result.

\begin{theorem} \label{result:theorem:ras_ar_equivalent}
There is a polynomial-time algorithm for Robust Appointment Scheduling if and only if there is a polynomial-time algorithm for the Airplane Refueling Problem.
\end{theorem}

The connection between the two problems allows using existing algorithms for the Airplane Refueling Problem to solve Robust Appointment Scheduling. 
Approximation algorithms, however, cannot be transferred due to the shift in the objective function. Still, as mentioned earlier, good approximation algorithms already exist for Robust Appointment Scheduling.
By Theorem~\ref{result:theorem:ar_bsp_equivalent}, Robust Appointment Scheduling is not only linked to Airplane Refueling but also to Block Stacking.  
In this way, they form a group of problems with a common computational complexity, albeit their true complexity remains unknown.
This demonstrates that advances in mathematical puzzles with no obvious application, such as the Airplane Refueling Problem and Block Stacking, can sometimes directly affect real-world problems such as Appointment Scheduling.

\section{A $\boldsymbol{(2+\epsilon)}$-Approximation for Block Stacking.} \label{section:approximation}

In Section~\ref{section:np_hardness}, we saw that Block Stacking with counterbalancing is NP-hard. A polynomial-time algorithm that finds an optimal solution is therefore elusive. Nevertheless, we can efficiently compute provably good solutions by resorting to approximation algorithms.
Throughout Sections~\ref{section:introduction} and~\ref{section:bsp}, we noted that  counterweights can facilitate greater overhangs. Often, however, the difference is small. 
Moreover, note that both with and without counterbalancing, we always want a wide and light block to protrude. 
As a result, in many cases, the same block protrudes in the optimal stack configurations of both scenarios, as, for example, in Figures~\ref{figure:treeby_blocks} and~\ref{figure:general_blocks}. 
Although this is not true for every set of blocks (see Figure~\ref{figure:remark}), we can prove that for any stack configuration, there is a fully right-aligned stack with the same protruding block that has at least half the original overhang.
In particular, this implies that finding the optimal fully right-aligned stack is a $2$-approximation of the general Block-Stacking Problem, as illustrated in Figure~\ref{figure:approximation}.
The example in Figure~\ref{figure:example_rearranging} demonstrates that this ratio is tight.

\begin{theorem} \label{result:theorem:two_approximation}
An optimal solution to Block Stacking without counterbalancing constitutes a \mbox{$2$-approximation} for the general Block-Stacking Problem.
\end{theorem}

\begin{proof}

Let $p \in B$ be the protruding block, and $\pi \colon B \to \{1,\dots,\vert B \vert \}$ the stacking order of an optimal stack configuration (with counterbalancing). Now consider the fully right-aligned stack with block $p$ on top, but other than that the same stacking order. Let $M_i$, $M^\prime_i$ be the aggregated masses with respect to the old and new stacking order, respectively, and note that $M_i = M^\prime_i$ for all $i$ with $\pi(i) > \pi(p)$. According to~\eqref{equation:bsp_no_counterbalancing}, the overhang of the new stack is
\begin{align*}
\sum_{i \in B} \frac{w_i m_i}{M^\prime_i} &= w_p +  \sum_{i \in B\setminus\{p\} } \frac{w_i m_i}{M^\prime_i} \\
&\geq w_p + \sum_{i :\, \pi(i) > \pi(p)} \frac{w_i m_i}{M^\prime_i} \\
&\geq \frac{w_p}{2} \Big(2 - \frac{m_p}{M_p}\Big) + \frac{1}{2} \sum_{i :\, \pi(i) > \pi(p)} \frac{w_i m_i}{M^\prime_i} \\
&= \frac{1}{2} \bigg( w_p \Big( 2 - \frac{m_p}{M_p} \Big) + \sum_{i :\, \pi(i) > \pi(p)} \frac{w_i m_i}{M_i} \bigg),
\end{align*}
i.e., at least half the overhang of the original stack configuration (cf.~\eqref{equation:overhang_flexible_configuration}).
\end{proof}

\begin{figure}[ht] 
\centering
\begin{subfigure}[b]{.32 \textwidth}
\centering
\includegraphics{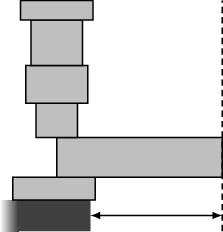}
\caption{}
\end{subfigure}
\begin{subfigure}[b]{.32 \textwidth}
\centering
\includegraphics{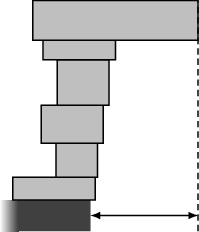}
\caption{}
\end{subfigure}
\begin{subfigure}[b]{.32 \textwidth}
\centering
\includegraphics{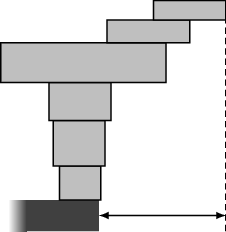}
\caption{}
\end{subfigure}
\caption{(a) is a stack configuration with maximum overhang. By Theorem~\ref{result:theorem:two_approximation}, (b) has at least half the overhang of (a). Therefore, the same holds for the best fully right-aligned stack (c).}
\label{figure:approximation}
\end{figure}

As a consequence, the PTAS for Block Stacking without counterbalancing (see Corollary~\ref{result:corollary:1_plus_e_approximation_no_counterbalancing}) automatically constitutes a $(2 + \epsilon)$-approximation algorithm for the general Block-Stacking Problem.

\begin{corollary} \label{result:corollary:2_plus_e_approximation}
For any $\epsilon > 0$, there is a $(2+\epsilon)$-approximation algorithm for the general Block-Stacking Problem.
\end{corollary}

\section{Conclusion.} \label{section:conclusion}

We presented several novel aspects of the Block-Stacking Problem. First, we saw that constructing a stack with the largest possible overhang is NP-hard. In addition, we discovered that Block Stacking is more than just a nice puzzle in its own right, as stacks without counterweights are closely related to the Airplane Refueling Problem and the real-world challenge of Appointment Scheduling. It remains an intriguing open question whether this collection of problems can be solved efficiently. Nevertheless, this connection made it possible to obtain a polynomial-time approximation scheme for the Block-Stacking Problem without counterbalancing and a $(2+\epsilon)$-approximation algorithm for the general case. 




 


\begin{thebibliography}{99}

\bibitem{bauerhenne2024robust} Bauerhenne, C., Kolisch, R., Schulz, A. S. (2024). Robust appointment scheduling with waiting time guarantees. \textit{arXiv.} doi.org/10.48550/arXiv.2402.12561

\bibitem{fine1947jeep} Fine, N. J. (1947). The jeep problem. \textit{Amer. Math. Monthly.} 54(1): 24-31. doi.org/10.2307/2304923

\bibitem{franklin1960range} Franklin, J. N. (1960). The range of a fleet of aircraft. \textit{J. Soc. Ind. Appl. Math.} 8(3): 541-548. doi.org/10.1137/0108039

\bibitem{gale1970jeep} Gale, D. (1970). The jeep once more, or jeeper by the dozen. \textit{Amer. Math. Monthly.} 77(5): 493-501. doi.org/10.1080/00029890.1970.11992525

\bibitem{garey1979computers} Garey, M. R., Johnson, D. S. (1979). \textit{Computers and Intractability: A Guide to the Theory of NP-Completeness.} San Francisco: W. H. Freeman.

\bibitem{gamzu2019polynomial} Gamzu, I., Segev, D. (2019). A polynomial-time approximation scheme for the airplane refueling problem. \textit{J. Scheduling.} 22(1): 119-135. doi.org/10.1007/s10951-018-0569-x

\bibitem{hall2005fun} Hall, J. F. (2005). Fun with stacking blocks. \textit{Amer. J. Physics.} 73(12): 1107-1116. doi.org/10.1119/1.2074007

\bibitem{hausrath1995gale} Hausrath, A., Jackson, B., Mitchem, J., Schmeichel, E. (1995). Gale's round-trip jeep problem. \textit{Amer. Math. Monthly.} 102(4): 299-309. doi.org/10.1080/00029890.1995.11990575

\bibitem{hohn2015performance} H{\"o}hn, W., Jacobs, T. (2015). On the performance of Smith's rule in single-machine scheduling with nonlinear cost. \textit{ACM Trans. Algorithms.} 11(4): 1-30. doi.org/10.1145/2629652

\bibitem{horton1997leaning} Horton, G. K., Holton, B. E., Freidkin, E. (1997). The leaning tower of tiles - revisited. \textit{Phys. Teach.} 35(4): 214-219. doi.org/10.1119/1.2344654

\bibitem{johnson1955leaning} Johnson, P. B. (1955). Leaning tower of lire. \textit{Amer. J. Phys.} 23: 240. doi.org/10.1119/1.1933957
		 
\bibitem{kazachkov2017stack} Kazachkov, A., Kire{\v s}, M. (2017). A stack of cards rebuilt with calculus. \textit{Phys. Educ.} 52(4): 045019. doi.org/10.1088/1361-6552/aa6a4e

\bibitem{korf2022jeep} Korf, R. E. (2022) A Jeep Crossing a Desert of Unknown Width, \textit{Amer. Math. Monthly.} 129(5): 435-444. doi.org/10.1080/00029890.2022.2051404

\bibitem{lengvarszky2023maximum} Lengv{\'a}rszky, Z. and Shepherd, D. (2023). Maximum overhang of sticky stacks. \textit{Math. Mag.} 96(5): 540-550. doi.org/10.1080/0025570X.2023.2266324

\bibitem{li2019fast} Li, J., Hu, X., Luo, J., Cui, J. (2019). A fast exact algorithm for airplane refueling problem. \textit{International Conference on Combinatorial Optimization and Applications.} Springer. pp. 316-327. doi.org/10.1007/978-3-030-36412-0\_25

\bibitem{megow2018dual} Megow, N., Verschae, J. (2018). Dual techniques for scheduling on a machine with varying speed. \textit{SIAM J. Discrete Math.} 32(3): 1541-1571. doi.org/10.1137/16M105589X

\bibitem{mittal2014robust} Mittal, S., Schulz, A. S., Stiller, S. (2014). Robust appointment scheduling. \textit{Approximation, Randomization, and Combinatorial Optimization. Algorithms and Techniques (APPROX/RANDOM 2014).} Schloss Dagstuhl, Leibniz-Zentrum fuer Informatik. doi.org/10.4230/LIPIcs.APPROX-RANDOM.2014.356

\bibitem{paterson2009maximum} Paterson, M., Peres, Y., Thorup, M., Winkler, P., Zwick, U. (2009). Maximum overhang. \textit{Amer. Math. Monthly.} 116(9): 763-787. doi.org/10.4169/000298909X474855

\bibitem{paterson2009overhang} Paterson, M., Zwick, U. (2009). Overhang. \textit{Amer. Math. Monthly.} 116(1): 19-44. doi.org/10.1080/00029890.2009.11920907

\bibitem{phipps1947jeep} Phipps, C. G. (1947). The jeep problem: a more general solution. \textit{Amer. Math. Monthly.} 54(8): 458-462. doi.org/10.1080/00029890.1947.11991865

\bibitem{polster2012case} Polster, B., Ross, M., Treeby, D. (2012). A case of continuous hangover. \textit{Amer. Math. Monthly.} 119(2): 122-139. doi.org/10.4169/amer.math.monthly.119.02.122

\bibitem{schulz2019robust} Schulz, A. S., Udwani, R. (2019). Robust appointment scheduling with heterogeneous costs. \textit{Approximation, Randomization, and Combinatorial Optimization. Algorithms and Techniques (APPROX/RANDOM 2019).} Schloss Dagstuhl, Leibniz-Zentrum fuer Informatik. doi.org/10.4230/LIPIcs.APPROX-RANDOM.2019.25

\bibitem{stiller2010increasing} Stiller, S., Wiese, A. (2010). Increasing speed scheduling and flow scheduling. \textit{International Symposium on Algorithms and Computation.} Springer. pp. 279-290 doi.org/10.1007/978-3-642-17514-5\_24

\bibitem{treeby2018further} Treeby, D. (2018). Further thoughts on a paradoxical tower. \textit{Amer. Math. Monthly.} 125(1): 44-60. doi.org/10.1080/00029890.2018.1390375

\bibitem{vasquez2015airplane} V\'asquez, O. C. (2015). For the airplane refueling problem local precedence implies global precedence. \textit{Optim. Lett.} 9(4): 663-675. doi.org/10.1007/s11590-014-0758-2

\bibitem{woeginger2010airplane} Woeginger, G. J. (2010). The airplane refueling problem. \textit{Seminar 10071 ``Scheduling''.} Schloss Dagstuhl - Leibniz-Zentrum fuer Informatik. p. 24. doi.org/10.4230/DagSemProc.10071.3

\bibitem{zhang2021novel} Zhang, G. J., Cui, J. C. (2021). A novel MILP model for $N$-vehicle exploration problem. \textit{J. Oper. Res. Soc. China.} 9(2): 359-373. doi.org/10.1007/s40305-019-00289-2

\end{thebibliography}
\end{document}